\documentclass{article}

\usepackage{cmbright}
\usepackage{amssymb,amsmath,amsthm}
\usepackage{mathrsfs}
\usepackage{epsfig}
\usepackage[usenames,dvipsnames]{color}

\def\B{\mathscr B}
\def\C{\mathbb C}
\def\d{\mathrm{d}}
\def\D{\mathscr D}

\def\H{\mathcal H}

\def\N{\mathbb N}

\def\R{\mathbb R}

\def\U{\mathscr U}

\def\Z{\mathbb Z}

\def\dom{\mathcal D}

\def\e{\mathop{\mathrm{e}}\nolimits}
\def\id{\mathop{\mathrm{id}}\nolimits}

\def\ltwo{\mathsf{L}^{\:\!\!2}}
\def\lone{\mathsf{L}^{\:\!\!1}}
\DeclareMathOperator*{\slim}{s\;\!-lim\;\!}

\newtheorem{Theorem}{Theorem}[section]
\newtheorem{Remark}[Theorem]{Remark}

\newtheorem{Corollary}[Theorem]{Corollary}
\newtheorem{Proposition}[Theorem]{Proposition}

\newtheorem{Example}[Theorem]{Example}

\begin{document}

%--------------------------------------------------------------------------------------
% Title
%--------------------------------------------------------------------------------------

\title{Commutator criteria for strong mixing II.\\More general and simpler}

\author{S. Richard$^1$\footnote{Supported by JSPS Grant-in-Aid for Young Scientists A
no 26707005.}~~and R. Tiedra de Aldecoa$^2$\footnote{Supported by the Chilean Fondecyt
Grant 1130168 and by the Iniciativa Cientifica Milenio ICM RC120002 ``Mathematical
Physics'' from the Chilean Ministry of Economy.}}

\date{\small}
\maketitle
\vspace{-1cm}

\begin{quote}
\emph{
\begin{itemize}
\item[$^1$] Graduate school of mathematics, Nagoya University,
Chikusa-ku, Nagoya 464-8602, Japan; On leave of absence from
Universit\'e de Lyon; Universit\'e
Lyon 1; CNRS, UMR 5208, Institut Camille Jordan,
43 blvd du 11 novembre 1918, F-69622
Villeurbanne-Cedex, France
\item[$^2$] Facultad de Matem\'aticas, Pontificia Universidad Cat\'olica de Chile,\\
Av. Vicu\~na Mackenna 4860, Santiago, Chile
\item[] \emph{E-mails:} richard@math.nagoya-u.ac.jp, rtiedra@mat.puc.cl
\end{itemize}
}
\end{quote}

%--------------------------------------------------------------------------------------

\begin{abstract}
We present a new criterion, based on commutator methods, for the strong mixing
property of unitary representations of topological groups equipped with a proper
length function. Our result generalises and unifies recent results on the strong
mixing property of discrete flows $\{U^N\}_{N\in\Z}$ and continuous flows
$\{\e^{-itH}\}_{t\in\R}$ induced by unitary operators $U$ and self-adjoint operators
$H$ in a Hilbert space. As an application, we present a short alternative proof (not
using convolutions) of the strong mixing property of the left regular representation
of $\sigma$-compact locally compact groups.
\end{abstract}


\smallskip


%--------------------------------------------------------------------------------------
\section{Introduction}
\setcounter{equation}{0}
%--------------------------------------------------------------------------------------

In the recent paper \cite{Tie15_2}, itself motivated by the previous papers
\cite{FRT13,Tie14,Tie12,Tie15}, it has been shown that commutator methods for unitary
and self-adjoint operators can be used to establish strong mixing. The main results of
\cite{Tie15_2} are the following two commutator criteria for strong mixing. First,
given a unitary operator $U$ in a Hilbert space $\H$, assume there exists an auxiliary
self-adjoint operator $A$ in $\H$ such that the commutators $[A,U^N]$ exist and are
bounded in some precise sense, and such that the strong limit
\begin{equation}\label{D_1}
D_1:=\slim_{N\to\infty}\frac1N[A,U^N]U^{-N}
\end{equation}
exists. Then, the discrete flow $\{U^N\}_{N\in\Z}$ is strongly mixing in
$\ker(D_1)^\perp$. Second, given a self-adjoint operator $H$ in $\H$, assume there
exists an auxiliary self-adjoint operator $A$ in $\H$ such that the commutators
$[A,\e^{-itH}]$ exist and are bounded in some precise sense, and such that the strong
limit
\begin{equation}\label{D_2}
D_2:=\slim_{t\to\infty}\frac1t[A,\e^{-itH}]\e^{itH}
\end{equation}
exists. Then, the continuous flow $\{\e^{-itH}\}_{t\in \R}$ is strongly mixing in
$\ker(D_2)^\perp$. These criteria were then applied to skew products of compact Lie
groups, Furstenberg-type transformations, time changes of horocycle flows and
adjacency operators on graphs.

The purpose of this note is to unify these two commutator criteria into a single, more
general, commutator criterion for strong mixing of unitary representations of
topological groups, and also to remove an unnecessary invariance assumption made in
\cite{Tie15_2}.

Our main result is the following. We consider a topological group $X$ equipped with a
proper length function $\ell:X\to\R_+$, a unitary representation $U:X\to\U(\H)$, and a
net $\{x_j\}_{j\in J}$ in $X$ with $x_j\to\infty$ (see Section \ref{sec_comm} for
precise definitions). Also, we assume there exists an auxiliary self-adjoint operator
$A$ in $\H$ such that the commutators $[A,U(x_j)]$ exist and are bounded in some
precise sense, and such that the strong limit
\begin{equation}\label{D_3}
D:=\slim_j\frac1{\ell(x_j)}[A,U(x_j)]U(x_j)^{-1}
\end{equation}
exists. Then, under these assumptions we show that the unitary representation $U$ is
strongly mixing in $\ker(D)^\perp$ along the net $\{x_j\}_{j\in J}$ (Theorem
\ref{Theorem_groups}). As a corollary, we obtain criteria for strong mixing in the
cases of unitary representations of compactly generated locally compact Hausdorff
groups (Corollary \ref{Corol_generated}) and the Euclidean group $\R^d$ (Corollary
\ref{Corol_R^d}). These results generalise the commutator criteria of \cite{Tie15_2}
for the strong mixing of discrete and continuous flows, as well as the strong limit
\eqref{D_3} generalises the strong limits \eqref{D_1} and \eqref{D_2} (see Remarks
\ref{rem_Z} and \ref{rem_R}). To conclude, we present in Example \ref{ex_regular} an
application which was not possible to cover with the results of \cite{Tie15_2}: a
short alternative proof (not using convolutions) of the strong mixing property of the
left regular representation of $\sigma$-compact locally compact Hausdorff groups.

We refer the reader to \cite{BR88,CCLTV11,HM79,LM92,Sch84,Zim84} for references on
strong mixing properties of unitary representations of groups.\\

\noindent
{\bf Acknowledgements.} The second author is grateful for the support and the
hospitality of the Graduate School of Mathematics of Nagoya University in March and
April 2015.

%--------------------------------------------------------------------------------------
\section{Commutator criteria for strong mixing}\label{sec_comm}
\setcounter{equation}{0}
%--------------------------------------------------------------------------------------

We start with a short review of basic facts on commutators of operators and regularity
classes associated with them. We refer to \cite[Chap.~5-6]{ABG96} for more details.

Let $\H$ be an arbitrary Hilbert space with scalar product
$\langle\;\!\cdot\;\!,\;\!\cdot\;\!\rangle$ antilinear in the first argument, denote
by $\B(\H)$ the set of bounded linear operators on $\H$, and write $\|\cdot\|$ both
for the norm on $\H$ and the norm on $\B(\H)$. Let $A$ be a self-adjoint operator in
$\H$ with domain $\dom(A)$, and take $S\in\B(\H)$. For any $k\in\N$, we say that $S$
belongs to $C^k(A)$, with notation $S\in C^k(A)$, if the map
\begin{equation}\label{eq_group}
\R\ni t\mapsto\e^{-itA}S\e^{itA}\in\B(\H)
\end{equation}
is strongly of class $C^k$. In the case $k=1$, one has $S\in C^1(A)$ if and only if
the quadratic form
$$
\dom(A)\ni\varphi\mapsto\big\langle\varphi,iSA\hspace{1pt}\varphi\big\rangle
-\big\langle A\hspace{1pt}\varphi,iS\varphi\big\rangle\in\C
$$
is continuous for the topology induced by $\H$ on $\dom(A)$. We denote by $[iS,A]$ the
bounded operator associated with the continuous extension of this form, or
equivalently the strong derivative of the map \eqref{eq_group} at $t=0$. Moreover, if
we set $A_\varepsilon:=(i\varepsilon)^{-1}(\e^{i\varepsilon A}-1)$ for
$\varepsilon\in\R\setminus\{0\}$, we have (see \cite[Lemma~6.2.3(a)]{ABG96}):
\begin{equation}\label{eq_com_1}
\slim_{\varepsilon\searrow0}[iS,A_\varepsilon]=[iS,A].
\end{equation}

Now, if $H$ is a self-adjoint operator in $\H$ with domain $\dom(H)$ and spectrum
$\sigma(H)$, we say that $H$ is of class $C^k(A)$ if $(H-z)^{-1}\in C^k(A)$ for some
$z\in\C\setminus\sigma(H)$. In particular, $H$ is of class $C^1(A)$ if and only if the
quadratic form
$$
\dom(A)\ni\varphi\mapsto
\big\langle\varphi,(H-z)^{-1}A\hspace{1pt}\varphi\big\rangle
-\big\langle A\hspace{1pt}\varphi,(H-z)^{-1}\varphi\big\rangle\in\C
$$
extends continuously to a bounded form with corresponding operator denoted by
$[(H-z)^{-1},A]\in\B(\H)$. In such a case, the set $\dom(H)\cap\dom(A)$ is a core for
$H$ and the quadratic form
$$
\dom(H)\cap\dom(A)\ni\varphi\mapsto\big\langle H\varphi,A\hspace{1pt}\varphi\big\rangle
-\big\langle A\hspace{1pt}\varphi,H\varphi\big\rangle\in\C
$$
is continuous in the topology of $\dom(H)$ (see \cite[Thm.~6.2.10(b)]{ABG96}). This
form then extends uniquely to a continuous quadratic form on $\dom(H)$ which can be
identified with a continuous operator $[H,A]$ from $\dom(H)$ to the adjoint space
$\dom(H)^*$. In addition, the following relation holds in $\B(\H)$ (see
\cite[Thm.~6.2.10(b)]{ABG96}):
\begin{equation}\label{eq_com_2}
[(H-z)^{-1},A]=-(H-z)^{-1}[H,A](H-z)^{-1}.
\end{equation}

With this, we can now present our first result, which is at the root of the new
commutator criterion for strong mixing. For it, we recall that a net
$\{x_j\}_{j\in J}$ in a topological space $X$ diverges to infinity, with notation
$x_j\to\infty$, if $\{x_j\}_{j\in J}$ has no limit point in $X$. This implies that for
each compact set $K\subset X$, there exists $j_K\in J$ such that $x_j\notin K$ for
$j\ge j_K$. In particular, $X$ is not compact. We also fix the notations $\U(\H)$ for
the set of unitary operators on $\H$ and $\R_+:=[0,\infty)$.

\begin{Proposition}\label{Prop_general}
Let $\{U_j\}_{j\in J}$ be a net in $\U(\H)$, let $\{\ell_j\}_{j\in J}\subset\R_+$
satisfy $\ell_j\to\infty$, assume there exists a self-adjoint operator $A$ in $\H$
such that $U_j\in C^1(A)$ for each $j\in J$, and suppose that the strong limit
$$
D:=\slim_j\frac1{\ell_j}[A,U_j]U_j^{-1}
$$
exists. Then, $\lim_j\big\langle\varphi,U_j\psi\big\rangle=0$ for all
$\varphi\in\ker(D)^\perp$ and $\psi\in\H$.
\end{Proposition}

Before the proof, we note that for $j\in J$ large enough (so that $\ell_j\neq 0$) the
operators $\frac1{\ell_j}[A,U_j]U_j^{-1}$ are well-defined, bounded and self-adjoint.
Therefore, their strong limit $D$ is also bounded and self-adjoint.

\begin{proof}
Let $\varphi=D\widetilde\varphi\in D\hspace{1pt}\dom(A)$ and $\psi\in\dom(A)$,
take $j\in J$ large enough, and set
$$
D_j:=\frac1{\ell_j}[A,U_j]U_j^{-1}.
$$
Since $U_j$ and $U_j^{-1}$ belong to $C^1(A)$ (see \cite[Prop.~5.1.6(a)]{ABG96}), both
$U_j \psi$ and $U_j^{-1}\widetilde\varphi$ belong to $\dom(A)$. Thus,
\begin{align*}
&\big|\big\langle\varphi,U_j\psi\big\rangle\big|\\
&=\big|\big\langle(D-D_j)\widetilde\varphi,U_j\psi\big\rangle
+\big\langle D_j\widetilde\varphi,U_j\psi\big\rangle\big|\\
&\le\big\|(D-D_j)\widetilde\varphi\big\|\!\;\|\psi\|
+\frac1{\ell_j}\big|\big\langle\big[A,U_j\big]U_j^{-1}\widetilde\varphi,
U_j\psi\big\rangle\big|\\
&\le\big\|(D-D_j)\widetilde\varphi\big\|\;\!\|\psi\|
+\frac1{\ell_j}\big|\big\langle A\widetilde\varphi,U_j\psi\big\rangle\big|
+\frac1{\ell_j}\big|\big\langle U_jAU_j^{-1}\widetilde\varphi,
U_j\psi\big\rangle\big|\\
&\le\big\|(D-D_j)\widetilde\varphi\big\|\;\!\|\psi\|
+\frac1{\ell_j}\big\|A\widetilde\varphi\big\|\;\!\|\psi\|
+\frac1{\ell_j}\big\|\widetilde\varphi\big\|\;\!\|A\psi\|.
\end{align*}
Since $D=\slim_jD_j$ and $\ell_j\to\infty$, we infer that
$\lim_j\big\langle\varphi,U_j\psi\big\rangle=0$, and thus the claim follows by the
density of $D\hspace{1pt}\dom(A)$ in $\overline{D\hspace{1pt}\H}=\ker(D)^\perp$ and
the density of $\dom(A)$ in $\H$.
\end{proof}

In the sequel, we assume that the unitary operators $U_j$ are given by a unitary
representation of a topological group $X$. We also assume that the scalars $\ell_j$
are given by a proper length function on $X$, that is, a function $\ell:X\to\R_+$
satisfying the following properties ($e$ denotes the identity of $X$):
\begin{enumerate}
\item[(L1)] $\ell(e)=0$,
\item[(L2)] $\ell(x^{-1})=\ell(x)$ for all $x\in X$,
\item[(L3)] $\ell(xy)\le\ell(x)+\ell(y)$ for all $x,y\in X$,
\item[(L4)] if $K\subset\R_+$ is compact, then $\ell^{-1}(K)\subset X$ is
relatively compact.
\end{enumerate}

\begin{Remark}[Topological groups with a proper left-invariant pseudo-metric]\label{remark_pseudo}
Let $X$ be a Hausdorff topological group equipped with a proper left-invariant
pseudo-metric $d:X\times X\to\R_+$ (see \cite[Def.~2.A.5 \& 2.A.7]{dd14}). Then,
simple calculations show that the associated length function $\ell:X\to\R_+$ given by
$\ell(x):=d(e,x)$ satisfies the properties (L1)-(L4) above. Examples of groups
admitting a proper left-invariant pseudo-metric are $\sigma$-compact locally compact
Hausdorff groups \cite[Prop.~4.A.2]{dd14}, as for instance compactly generated locally
compact Hausdorff groups with the word metric \cite[Prop.~4.B.4(2)]{dd14}.
\end{Remark}

The next theorem provides a general commutator criterion for the strong mixing
property of a unitary representation of a topological group. Before stating it, we
recall that if a topological group $X$ is equipped with a proper length function
$\ell$, and if $\{x_j\}_{j\in J}$ is a net in $X$ with $x_j\to\infty$, then
$\ell(x_j)\to\infty$ (this can be shown by absurd using the property (L4) above).

\begin{Theorem}[Topological groups]\label{Theorem_groups}
Let $X$ be a topological group equipped with a proper length function $\ell$, let
$U:X\to\U(\H)$ be a unitary representation of $X$, let $\{x_j\}_{j\in J}$ be a net in
$X$ with $x_j\to\infty$, assume there exists a self-adjoint operator $A$ in $\H$ such
that $U(x_j)\in C^1(A)$ for each $j\in J$, and suppose that the strong limit
\begin{equation}\label{def_de_D}
D:=\slim_j\frac1{\ell(x_j)}[A,U(x_j)]U(x_j)^{-1}
\end{equation}
exists. Then,
\begin{enumerate}
\item[(a)] $\lim_j\big\langle\varphi,U(x_j)\psi\big\rangle=0$ for all
$\varphi\in\ker(D)^\perp$ and $\psi\in\H$,
\item[(b)] $U$ has no nontrivial finite-dimensional unitary subrepresentation in
$\ker(D)^\perp$.
\end{enumerate}
\end{Theorem}

\begin{proof}
The claim (a) follows from Proposition \ref{Prop_general} and the fact that
$\ell(x_j)\to\infty$. The claim (b) follows from (a) and the fact that matrix
coefficients of finite-dimensional unitary representations of a group do not vanish at
infinity (see for instance \cite[Rem.~2.15(iii)]{BM00}).
\end{proof}

\begin{Remark}
(i) The result of Theorem \ref{Theorem_groups}(a) amounts to the strong mixing
property of the unitary representation $U$ in $\ker(D)^\perp$ along the net
$\{x_j\}_{j\in J}$, as mentioned in the introduction. If the strong limit
\eqref{def_de_D} exists for all nets $\{x_j\}_{j\in J}$ with $x_j\to\infty$, then
Theorem \ref{Theorem_groups}(a) implies the usual strong mixing property of the
unitary representation $U$ in $\ker(D)^\perp$.

(ii) One can easily see that Theorem \ref{Theorem_groups} remains true if the scalars
$\ell(x_j)$ in \eqref{def_de_D} are replaced by $(f\circ\ell)(x_j)$, with
$f:\R_+\to\R_+$ any proper function. For simplicity, we decided to present only the
case $f=\id_{\R_+}$, but we note this additional freedom might be useful in
applications.
\end{Remark}

Theorem \ref{Theorem_groups} and Remark \ref{remark_pseudo} imply the following result
in the particular case of a compactly generated locally compact group $X$:

\begin{Corollary}[Compactly generated locally compact groups]\label{Corol_generated}
Let $X$ be a compactly generated locally compact Hausdorff group with generating set
$Y$ and word length function $\ell$, let $U:X\to\U(\H)$ be a unitary representation of
$X$, let $\{x_j\}_{j\in J}$ be a net in $X$ with $x_j\to\infty$, assume there exists a
self-adjoint operator $A$ in $\H$ such that $U(y)\in C^1(A)$ for each $y\in Y$, and
suppose that the strong limit
\begin{equation}\label{D_generating}
D:=\slim_j\frac1{\ell(x_j)}[A,U(x_j)]U(x_j)^{-1}
\end{equation}
exists. Then,
\begin{enumerate}
\item[(a)] $\lim_j\big\langle\varphi,U(x_j)\psi\big\rangle=0$ for all
$\varphi\in\ker(D)^\perp$ and $\psi\in\H$,
\item[(b)] $U$ has no nontrivial finite-dimensional unitary subrepresentation in
$\ker(D)^\perp$.
\end{enumerate}
\end{Corollary}

\begin{proof}
In order to apply Theorem \ref{Theorem_groups}, we first note from Remark
\ref{remark_pseudo} that the word length function $\ell$ is a proper length function.
Second, we note that $X=\bigcup_{n\ge1}(Y\cup Y^{-1})^n$. Therefore, for each $x\in X$
there exist $n\ge1$, $y_1,\ldots,y_n\in Y$ and $m_1,\ldots,m_n\in\{\pm1\}$ such that
$x=y_1^{m_1}\cdots y_n^{m_n}$. Thus,
$$
U(x)
=U\big(y_1^{m_1}\cdots y_n^{m_n}\big)
=U(y_1)^{m_1}\cdots U(y_n)^{m_n},
$$
and it follows from the inclusions $U(y_1),\ldots,U(y_n)\in C^1(A)$ and standard
results on commutator methods \cite[Prop.~5.1.5 \& 5.1.6(a)]{ABG96} that
$U(x)\in C^1(A)$. Thus, we have $U(x_j)\in C^1(A)$ for each $j\in J$, and the
commutators $[A,U(x_j)]$ appearing in \eqref{D_generating} make sense. So, we can
apply Theorem \ref{Theorem_groups} to conclude.
\end{proof}

\begin{Remark}\label{rem_Z}
Corollary \ref{Corol_generated} is a generalisation of \cite[Thm.~3.1]{Tie15_2} to the
case of unitary representations of compactly generated locally compact Hausdorff
groups. Indeed, if we let $X$ be the additive group $\Z$ with generating element $1$,
take the trivial net $\{x_j=j\}_{j\in\N^*}=\{N\mid N\in\N^*\}$, and set $U:=U(1)$ in
Corollary \ref{Corol_generated}, then the strong limit \eqref{D_generating} reduces to
$$
D=\slim_{N\to\infty}\frac1N\big[A,U^N\big]U^{-N}
=\slim_{N\to\infty}\frac1N\sum_{n=0}^{N-1}U^n\big([A,U]U^{-1}\big)U^{-n},
$$
which is the strong limit appearing in \cite[Thm.~3.1]{Tie15_2}. In Corollary
\ref{Corol_generated} we also removed the unnecessary invariance assumption
$\eta(D)\hspace{1pt}\dom(A)\subset\dom(A)$ for each
$\eta\in C^\infty_{\rm c}(\R\setminus\{0\})$. So, the strong mixing properties for
skew products and Furstenberg-type transformations established in
\cite[Sec.~3]{Tie15_2} and \cite[Sec.~3]{CT15} can be obtained more directly using
Corollary \ref{Corol_generated}.
\end{Remark}

In the next corollary we consider the case of a strongly continuous unitary
representation $U:\R^d\to\U(\H)$ of the Euclidean group $\R^d$, $d\ge1$. In such a
case Stone's theorem implies the existence of a family of mutually commuting
self-adjoint operators $H_1,\ldots,H_d$ such that $U(x)=\e^{-i\sum_{k=1}^dx_kH_k}$ for
each $x=(x_1,\ldots,x_d)\in\R^d$. Therefore, we give a criterion for strong mixing in
terms of the operators $H_1,\ldots,H_d$. We use the shorthand notations
$$
H:=(H_1,\ldots,H_d),
\quad\Pi(H):=(H_1+i)^{-1}\cdots(H_d+i)^{-1}
\quad\hbox{and}\quad x\cdot H:=\sum_{k=1}^dx_k H_k.
$$

\begin{Corollary}[Euclidean group $\R^d$]\label{Corol_R^d}
Let $\R^d$, $d\ge1$, be the Euclidean group with Euclidean length function $\ell$, let
$U:\R^d\to\U(\H)$ be a strongly continuous unitary representation of $\R^d$, let
$\{x_j\}_{j\in J}$ be a net in $\R^d$ with $x_j\to\infty$, assume there exists a
self-adjoint operator $A$ in $\H$ such that $(H_k-i)^{-1}\in C^1(A)$ for each
$k\in\{1,\ldots,d\}$, and suppose that the strong limit
\begin{equation}\label{D_dimension_d}
D:=\slim_j\frac1{\ell(x_j)}\int_0^{1}\d s\,\e^{-is(x_j\cdot H)}
\Pi(H)\big[i(x_j\cdot H),A\big]\Pi(H)^*\e^{is(x_j\cdot H)}
\end{equation}
exists. Then,
\begin{enumerate}
\item[(a)] $\lim_j\big\langle\varphi,U(x_j)\psi\big\rangle=0$ for all
$\varphi\in\ker(D)^\perp$ and $\psi\in\H$,
\item[(b)] $U$ has no nontrivial finite-dimensional unitary subrepresentation in
$\ker(D)^\perp$.
\end{enumerate}
\end{Corollary}

\begin{proof}
The proof consists in applying Theorem \ref{Theorem_groups} with $A$ replaced by a
new operator $\widetilde A$ that we now define.

The inclusions $(H_1-i)^{-1},\ldots,(H_d-i)^{-1}\in C^1(A)$ and the standard result on
commutator methods \cite[Prop.~5.1.5]{ABG96} imply that $\Pi(H)^*\in C^1(A)$. So, we
have $\Pi(H)^*\hspace{1pt}\dom(A)\subset\dom(A)$, and the operator
$$
\widetilde A\hspace{1pt}\varphi
:=\Pi(H)A\hspace{1pt}\Pi(H)^*\varphi,\quad\varphi\in\dom(A),
$$
is essentially self-adjoint (see \cite[Lemma~7.2.15]{ABG96}). Take $\varphi\in\dom(A)$
and $j_0\in J$ such that $\ell(x_j)>0$ for all $j\ge j_0$, and define for
$\varepsilon\in\R\setminus\{0\}$ the operator
$A_\varepsilon:=(i\varepsilon)^{-1}(\e^{i\varepsilon A}-1)$. Then, we have
\begin{align}
&\big\langle\widetilde A\hspace{1pt}\varphi,U(x_j)\varphi\big\rangle
-\big\langle\varphi,U(x_j)\widetilde A\hspace{1pt}\varphi\big\rangle\nonumber\\
&=\lim_{\varepsilon\searrow0}\big(\big\langle\varphi,
\Pi(H)A_\varepsilon\Pi(H)^*\e^{-i(x_j\cdot H)}\varphi\big\rangle
-\big\langle\varphi,\e^{-i(x_j\cdot H)}\Pi(H)
A_\varepsilon\Pi(H)^*\varphi\big\rangle\big)\nonumber\\
&=\lim_{\varepsilon\searrow0}\int_0^{\ell(x_j)}\d q\,\frac\d{\d q}\;\!
\big\langle\varphi,\e^{i(q-\ell(x_j))(x_j\cdot H)/\ell(x_j)}\Pi(H)
A_\varepsilon\Pi(H)^*\e^{-iq(x_j\cdot H)/\ell(x_j)}\varphi\big\rangle\nonumber\\
&=\frac1{\ell(x_j)}\lim_{\varepsilon\searrow0}\int_0^{\ell(x_j)}\d q\,
\big\langle\varphi,\e^{i(q-\ell(x_j))(x_j\cdot H)/\ell(x_j)}\Pi(H)
\big[i(x_j\cdot H),A_\varepsilon\big]\Pi(H)^*\e^{-iq(x_j\cdot H)/\ell(x_j)}
\varphi\big\rangle.\label{eq_lim_int}
\end{align}
But, $(H_1-i)^{-1},\ldots,(H_d-i)^{-1}\in C^1(A)$. Therefore, \eqref{eq_com_1} and
\eqref{eq_com_2} imply that
$$
\slim_{\varepsilon\searrow0}\Pi(H)\big[i(x_j\cdot H),A_\varepsilon\big]\Pi(H)^*
=\Pi(H)\big[i(x_j\cdot H),A\big]\Pi(H)^*,
$$
and we can exchange the limit and the integral in \eqref{eq_lim_int} to obtain
\begin{align*}
&\big\langle\widetilde A\hspace{1pt}\varphi,U(x_j)\varphi\big\rangle
-\big\langle\varphi,U(x_j)\widetilde A\hspace{1pt}\varphi\big\rangle\\
&=\frac1{\ell(x_j)}\int_0^{\ell(x_j)}\d q\,\big\langle\varphi,
\e^{i(q-\ell(x_j))(x_j\cdot H)/\ell(x_j)}\Pi(H)\big[i(x_j\cdot H),A\big]\Pi(H)^*
\e^{-iq(x_j\cdot H)/\ell(x_j)}\varphi\big\rangle\\
&=\frac1{\ell(x_j)}\int_0^{\ell(x_j)}\d r\,\big\langle\varphi,
\e^{-ir(x_j\cdot H)/\ell(x_j)}\Pi(H)\big[i(x_j\cdot H),A\big]\Pi(H)^*
\e^{i(r-\ell(x_j))(x_j\cdot H)/\ell(x_j)}\varphi\big\rangle\\
&=\int_0^1\d s\,\big\langle\varphi,
\e^{-is(x_j\cdot H)}\Pi(H)\big[i(x_j\cdot H),A\big]\Pi(H)^*
\e^{is(x_j\cdot H)}U(x_j)\hspace{1pt}\varphi\big\rangle\\
&=\big\langle\varphi,\ell(x_j)D_jU(x_j)\hspace{1pt}\varphi\big\rangle
\end{align*}
with
\begin{align*}
D_j&:=\frac1{\ell(x_j)}\int_0^{1}\d s\,\e^{-is(x_j\cdot H)}
\Pi(H)\big[i(x_j\cdot H),A\big]\Pi(H)^*\e^{is(x_j\cdot H)}.
\end{align*}
Since $\dom(A)$ is a core for $\widetilde A$, this implies that
$U(x_j)\in C^1(\widetilde A)$ with
$\big[\widetilde A,U(x_j)\big]=\ell(x_j)D_j\hspace{1pt}U(x_j)$. Therefore, we have
$$
D_j=\frac1{\ell(x_j)}\big[\widetilde A,U(x_j)\big]U(x_j)^{-1},
$$
and all the assumptions of Theorem \ref{Theorem_groups} are satisfied with $A$
replaced by $\widetilde A$.
\end{proof}

\begin{Remark}\label{rem_R}
Corollary \ref{Corol_R^d} is a generalisation of \cite[Thm.~4.1]{Tie15_2} to the case
of strongly continuous unitary representations of $\R^d$ for an arbitrary $d\ge1$.
Indeed, if we set $d=1$, write $H$ for $H_1$, and take the trivial net
$\{x_j=j\}_{j\in(0,\infty)}=\{t\mid t>0\}$ in Corollary \ref{Corol_R^d}, then the
strong limit \eqref{D_dimension_d} reduces to
\begin{align*}
D&=\slim_{t\to\infty}\frac1t\int_0^1\d s\,\e^{-is(t\cdot H)}
\big(H+i\big)^{-1}\big[itH,A\big]\big(H-i\big)^{-1}\e^{is(t\cdot H)}\\
&=\slim_{t\to\infty}\frac1t\int_0^t\d s\,\e^{-isH}\big(H+i\big)^{-1}
\big[iH,A\big]\big(H-i\big)^{-1}\e^{isH},
\end{align*}
which is (up to a sign) the strong limit appearing in \cite[Thm.~4.1]{Tie15_2}. In
Corollary \ref{Corol_R^d}, we also removed the unnecessary invariance assumption
$\eta(D)\hspace{1pt}\dom(A)\subset\dom(A)$ for each
$\eta\in C^\infty_{\rm c}(\R\setminus\{0\})$. So, the strong mixing properties for
adjacency operators, time changes of horocycle flows, etc., established in
\cite[Sec.~4]{Tie15_2} can be obtained more directly using Corollary \ref{Corol_R^d}.
\end{Remark}

To conclude, we add to the list of examples presented in \cite{Tie15_2} an application
which was not possible to cover with the results of \cite{Tie15_2}. It is a short
alternative proof (not using convolutions) of the strong mixing property of the left
regular representation of $\sigma$-compact locally compact Hausdorff groups:

\begin{Example}[Left regular representation]\label{ex_regular}
Let $X$ be a $\sigma$-compact locally compact Hausdorff group with left Haar measure
$\mu$ and proper length function $\ell$ (see Remark \ref{remark_pseudo}). Let
$\D\subset\H$ be the set of functions $X\to\C$ with compact support, and let
$U:X\to\U(\H)$ be the left regular representation of $X$ on $\H:=\ltwo(X,\mu)$
given by
$$
U(x)\varphi:=\varphi(x^{-1}\;\!\cdot\;\!),\quad x\in X,~\varphi\in\H,
$$
Let finally $A$ be the maximal multiplication operator in $\H$ given by
$$
A\varphi:=\ell\;\!\varphi\equiv \ell(\cdot)\varphi,
\quad\varphi\in\dom(A):=\big\{\varphi\in\H\mid\|\ell\;\!\varphi\|<\infty\big\}.
$$
For $\varphi\in\D$ and $x\in X$, one has
$$
AU(x)\varphi-U(x)A\varphi
=\big(\ell(\;\!\cdot\;\!)-\ell(x^{-1}\;\!\cdot\;\!)\big)U(x)\varphi.
$$
Furthermore, the properties (L2)-(L3) of a length function imply that
\begin{equation}\label{eq_bound}
\big|\big(\ell(\;\!\cdot\;\!)-\ell(x^{-1}\;\!\cdot\;\!)\big)\big|\le\ell(x).
\end{equation}
Therefore, since $\D$ is dense in $\dom(A)$, it follows that $U(x)\in C^1(A)$ with
$$
[A,U(x)]U(x)^{-1}=\ell(\;\!\cdot\;\!)-\ell(x^{-1}\;\!\cdot\;\!).
$$
Now, we take $\{x_j\}_{j\in J}$ a net in $X$ with $x_j\to\infty$, and show that
\begin{equation}\label{eq_D_minus_1}
D:=\slim_j\frac1{\ell(x_j)}[A,U(x_j)]U(x_j)^{-1}=-1.
\end{equation}
For this, we first note that for $\varphi\in\H$ we have
$$
\left(\frac1{\ell(x_j)}[A,U(x_j)]U(x_j)^{-1}+1\right)\varphi
=\frac{\ell(\;\!\cdot\;\!)-\ell(x_j^{-1}\;\!\cdot\;\!)+\ell(x_j)}{\ell(x_j)}
\;\!\varphi.
$$
Next, we note that \eqref{eq_bound} implies that
$$
\left|\frac{\ell(\;\!\cdot\;\!)-\ell(x_j^{-1}\;\!\cdot\;\!)+\ell(x_j)}{\ell(x_j)}
\;\!\varphi\right|^2
\le4\;\!|\varphi|^2
\in\lone(X,\mu),
$$
and that the properties (L2)-(L3) imply that
$$
\lim_j\left|\frac{\ell(\;\!\cdot\;\!)-\ell(x_j^{-1}\;\!\cdot\;\!)+\ell(x_j)}
{\ell(x_j)}\;\!\varphi\right|^2
\le\lim_j\left|\frac{2\;\!\ell(\;\!\cdot\;\!)}{\ell(x_j)}\;\!\varphi\right|^2
=0
\quad\hbox{$\mu$-almost everywhere.}
$$
Therefore, we can apply Lebesgue dominated convergence theorem to get the equality
$$
\slim_j\left(\frac1{\ell(x_j)}[A,U(x_j)]U(x_j)^{-1}+1\right)\varphi=0,
$$
which proves \eqref{eq_D_minus_1}. So, Theorem \ref{Theorem_groups} applies with
$D=-1$, and thus $\lim_j\big\langle\varphi,U(x_j)\psi\big\rangle=0$ for all
$\varphi,\psi\in\H$.
\end{Example}

%--------------------------------------------------------------------------------------
%\bibliography{../bibliographie/bibliographie}
%--------------------------------------------------------------------------------------

\def\cprime{$'$} \def\polhk#1{\setbox0=\hbox{#1}{\ooalign{\hidewidth
  \lower1.5ex\hbox{`}\hidewidth\crcr\unhbox0}}}
  \def\polhk#1{\setbox0=\hbox{#1}{\ooalign{\hidewidth
  \lower1.5ex\hbox{`}\hidewidth\crcr\unhbox0}}}
  \def\polhk#1{\setbox0=\hbox{#1}{\ooalign{\hidewidth
  \lower1.5ex\hbox{`}\hidewidth\crcr\unhbox0}}} \def\cprime{$'$}
  \def\cprime{$'$} \def\polhk#1{\setbox0=\hbox{#1}{\ooalign{\hidewidth
  \lower1.5ex\hbox{`}\hidewidth\crcr\unhbox0}}}
  \def\polhk#1{\setbox0=\hbox{#1}{\ooalign{\hidewidth
  \lower1.5ex\hbox{`}\hidewidth\crcr\unhbox0}}} \def\cprime{$'$}
  \def\cprime{$'$} \def\cprime{$'$}

%--------------------------------------------------------------------------------------

\end{document}